\documentclass[a4paper,12pt]{article}

\usepackage{fullpage}
\usepackage{mathtools}
\usepackage[utf8]{inputenc}
\usepackage{amsthm,amsmath,amstext, amssymb, xcolor, scrextend, tikz, multirow, enumerate, wasysym, makecell, tabularx, hyperref,  pbox, lipsum,csquotes,  tikz-cd, multicol, calc, accents, stackengine, kpfonts}
\usepackage[all]{xy}
\usepackage{xcolor}
\usepackage{palatino}
\usepackage{cancel}
\usepackage{authblk}
\usepackage{setspace}

\usepackage{mathtools}

\DeclareSymbolFont{AMSb}{U}{msb}{m}{n}
\DeclareMathSymbol{\N}{\mathbin}{AMSb}{"4E}
\DeclareMathSymbol{\Z}{\mathbin}{AMSb}{"5A}
\DeclareMathSymbol{\R}{\mathbin}{AMSb}{"52}
\DeclareMathSymbol{\Q}{\mathbin}{AMSb}{"51}
\DeclareMathSymbol{\I}{\mathbin}{AMSb}{"49}
\DeclareMathSymbol{\C}{\mathbin}{AMSb}{"43}

\DeclareFontFamily{U}{mathx}{\hyphenchar\font45}
\DeclareFontShape{U}{mathx}{m}{n}{<-> mathx10}{}
\DeclareSymbolFont{mathx}{U}{mathx}{m}{n}
\DeclareMathAccent{\widebar}{0}{mathx}{"73}

\def\lim{\mathop{\rm lim}\nolimits}

\def\Ext{\mathrm{Ext}}
\def\Hom{\mathrm{Hom}}

\def\Ker{\mathrm{Ker}}
\def\Ker{\mathrm{Ker}}
\def\ker{\mathop{\rm Ker}\nolimits}

\newcommand{\res}[1]{\hspace{-0.6mm}\downarrow_{\hspace{-0.25mm}{#1}}}

\newcommand{\F}{\mathbb{F}}

\DeclareMathOperator{\Kern}{Ker}

\usetikzlibrary{decorations.shapes}
\usetikzlibrary{decorations.text}
\usetikzlibrary{calc}

\hypersetup{colorlinks=true,linkcolor=blue,citecolor=blue, filecolor=blue,urlcolor=blue}

\newtheorem{theorem}{Theorem}[section]
\newtheorem{lemma}[theorem]{Lemma}

\title{Permutation modules over cyclic $p$-groups}
\author[a]{Marlon Estanislau}
\affil[a]{Universidade Federal de Minas Gerais, Belo Horizonte, MG, Brazil}

\begin{document}

\maketitle

\footnotetext[1]{\textit{Email addresses:}  mestanislau@ufmg.br}
\footnotetext[2]{Research is part of the author's PhD thesis at the Federal University of Minas Gerais, supervised by 
John William MacQuarrie.}

\begin{abstract}

Let $G$ be a cyclic $p$-group for some prime number $p>0$ and let $R$ be a complete discrete valuation ring in mixed characteristic. In this paper, we present a generalization of two results that characterize $RG$-permutation modules, extending previous work by B. Torrecillas and Th. Weigel. Their original results were established under the assumption that $ p$ is unramified in $R$, whereas we extend their characterization to the case where $p$ may be ramified. Unlike prior approaches, our proofs rely solely on fundamental facts from group cohomology and a version of Weiss' Theorem, avoiding deeper categorical techniques. 

\end{abstract}



\vspace{1.5mm}
\textbf{keywords}:
finite $p$-groups, permutation modules, lattices, $p$-adic representations.

2020 MSC: code  20c11.

\section{Introduction}
Throughout this article, $G$ is a  cyclic $p$-group for some prime number $p>0$, $N$ is a normal subgroup of $G$ (denoted $N\lhd G$),  and $R$ is  a \emph{complete discrete valuation ring in mixed characteristic} (that is, with fraction field of characteristic $0$ and 
residue field of characteristic $p$). Let $A$ denote a quotient ring of $R$. In this work, we will deal with    $AG$-\emph{lattices}, which are  $AG$-modules that are $A$-free and finitely  generated. 

  Of particular importance are the $AG$-\emph{permutation modules}: $AG$-lattices having an $A$-basis that is preserved set-wise by the multiplication of $G$.  Such lattices are extremely well behaved and serve as a natural generalization of free modules.  However, given a lattice, it may not be clear whether it is a permutation module. In 1988, A. Weiss provided the following criterion for permutation modules, which is of key importance, having applications in number theory \cite{WeissAnnals} and block theory \cite{Puig},  to name a few:
 \begin{theorem}(\cite[Theorem 2]{WeissAnnals})\label{t1}
      Let $G$ be a finite $p$-group and let $\Z_p$ denote the ring of $p$-adic integers. Let $U$ be a $\Z_pG$-lattice and  suppose there is a normal subgroup  $N$ of $G$ such that:
     \begin{enumerate}
    \item the module $U$ restricted to $N$ is  $\Z_pN$-free,
    \item the submodule of $N$-fixed points $U^N$ is a $\Z_p[G/N]$-permutation module.
   
\end{enumerate}
 Then $U$ itself is a $\Z_pG$-permutation module.
 \end{theorem}
 Later in 1993, Weiss generalized this result for finite extensions of $\Z_p$ and recently J.W. MacQuarrie, P. Symonds and P.A. Zalesskii generalized Weiss' Theorem to the following result:
\begin{theorem}(Special case of \cite[Theorem 1.2]{MACQUARRIE2020106925})\label{tjpp}
     Let $R$ be a complete discrete valuation ring in mixed characteristic with 
residue field of characteristic $p$, let $G$ be a finite $p$-group and let $U$ be an $RG$-lattice. Suppose there is a normal subgroup $N$ of $G$ such that:
\begin{enumerate}
    \item the module $U$ restricted to $N$ is  $RN$-free,
    \item the submodule of $N$-fixed points $U^N$ is an $R[G/N]$-permutation module.
   
\end{enumerate}
 Then $U$ itself is an $RG$-permutation module.
\end{theorem}
  This last result has been applied  in the study of profinite groups (cf.\cite{P}) and calculation of Picard groups of blocks of finite groups (cf.\cite{BOLTJE202070}, \cite{Florian}, \cite{LIVESEY202194}) and will be fundamental for us.

 Weiss' Theorem cannot be applied to lattices for which $U\res{N}$ is not free.  If $U$ is a permutation module, then of course $U\res{N}$ is necessarily a permutation module, but, in general, a sensible analogous statement to Weiss's Theorem demanding only that $U\res{N}$ be a permutation module seems far from obvious. In 2013, B. Torrecillas and Th. Weigel gave the following generalization of Weiss's Theorem  when $G$ is cyclic and $pR$ is maximal in $R$: 
   \begin{theorem}(\cite[Proposition 6.12]{TORRECILLAS2013533})\label{Torr.Wei}
    Let $R$ be a complete discrete valuation ring in mixed characteristic with maximal ideal $pR$, let $G$ be a cyclic $p$-group and  $U$ be an $RG$ lattice. Suppose  there is a normal subgroup  $N$  of $G$ such that:
    \begin{enumerate}
        \item the lattice $U$ restricted to $N$ is an $RN$-permutation module,
        \item  the submodule of $N$-fixed points $U^N$  is an $R[G/N]$-permutation module.
    \end{enumerate}
    Then $U$ itself is an $RG$-permutation module. 
  \end{theorem}
   The purpose of this paper is to give a generalization of Theorem \ref{Torr.Wei} for  arbitrary complete discrete valuation ring in mixed characteristic: 
   \begin{theorem}\label{myresult}
    Let $R$ be a complete discrete valuation ring in mixed characteristic with residue field of characteristic $p$, let $G$ be a cyclic $p$-group and let $U$ be an $RG$-lattice. Suppose there is a subgroup $N$ of $G$ such that:
     \begin{enumerate}
        \item the lattice $U$ restricted to $N$ is an $RN$-permutation module,
        \item  the submodule of $N$-fixed points $U^N$  is an $R[G/N]$-permutation module.
    \end{enumerate}
    Then $U$ itself is an $RG$-permutation module.
  \end{theorem}

  It is worth noting that  a characterization has already been obtained by MacQuarrie and Zalesskii when $R$ is the ring of $p$-adic integers $\Z_p$ and $G$ is a general $p$-group \cite[Theorem 2]{John}. But in \cite[Theorem 2]{John}, much stronger hypotheses were imposed.
   
   As usual, we denote by $\text{H}^1(G,-)$ the first right derived functor of the fixed point
   functor $(-)^G$.      
    Let $G$ be a finite $p$-group. An $RG$-module $U$ is called $G$-\emph{coflasque} if $\text{H}^1(L,U)=0$ for every subgroup $L$ of $G$. If $U$ is an $RG$-permutation module then $\text{H}^1(G,U)=0$ \cite[Lemma 1.5]{Gregory}, and hence (since the restriction of a permutation module to a subgroup is a permutation module) $RG$-permutation modules are $G$-coflasque. However, in general,  there are $RG$-lattices that are $G$-coflasque but are not $RG$-permu\-tation modules. But when $G$ is a cyclic $p$-group and $p$ is prime in $R$, then it is sufficient to ensure that an $RG$-lattice is $G$-coflasque  for it to be an $RG$-permutation module (cf.\cite{TORRECILLAS2013533}, Proposition 6.7). This fact is fundamental for the authors in \cite{TORRECILLAS2013533} to demonstrate Theorem \ref{Torr.Wei}. 
    
   We prove  Theorem \ref{myresult} in a different way from that in \cite{TORRECILLAS2013533}. In \cite{TORRECILLAS2013533}, the authors make  a detailed analysis of
   the category of cohomological $G$-Mackey functors with values in the category of $R$-modules.  We only  use  Theorem \ref{tjpp}  and basic facts about Group Cohomology.  We also get a generalization of Proposition $6.7$ in \cite{TORRECILLAS2013533} for complete discrete valuation rings in mixed characteristic for which   coflasque $RC_p$-modules are permutation modules where $C_p$ is the cyclic group of order $p$:
   \begin{theorem}\label{corollo}
 Let $R$ be a complete discrete valuation ring in mixed characteristic with residue field of characteristic $p$ such that  coflasque $RC_p$-lattices are $RC_p$-permutation  modules. Let $G$ be a finite  cyclic $p$-group and let $U$ be an $RG$-lattice. Then the following are equivalent:
 \begin{enumerate}
     \item $U$ is $RG$-permutation module,
     \item $U$ is $G$-coflasque.
     \item $U_N$ is an $R$-lattice for all subgroups $N$ of $G$,

     where $U_N$ is the largest quotient of $U$ in which $N$ acts trivially.
 \end{enumerate}
\end{theorem}
It should be noted that when $ p$ is unramified in $ R $, every  $RC_p$-coflasque module is an  $RC_p$-permutation module (cf.\cite{butler2}, $3.5(a)$). There are  cases where this remains true even when $ p$ is ramified in  $R$. For instance, this is the case when  $p = 2$, as follows from \cite[$3.5(b)$]{butler2} and \cite[Proposition~6.14]{MACQUARRIE2020106925}.

This paper is organized as follows. Section \ref{s2} introduces the necessary definitions and preliminary observations. Theorems \ref{myresult} and \ref{corollo} are proved in Section \ref{s3}.

\subsubsection*{Acknowledgements}
    The author thanks  John William MacQuarrie for helpful conversations and for his attention and help in preparing this manuscript.
    The author was supported by  CAPES Doctoral Grant 88887.688170/2022-00 and by FAPEMIG Doctoral Grant 13632/2025-00.

\section{Preliminaries}\label{s2}
 
    Let $\pi$ be a prime element of $R$ and  denote by  $\F$  the residue field $R/\pi R$. Throughout the text, $U$ will denote  an $RG$-module, and   we denote by $\overline{U}$ the corresponding $\F G$-module  $U/\pi U$. We will write at times 
   $\overline{X}$ when $X$ is a submodule of $U$. In our context, $X$ will always  be an $R$-direct summand of  $U$, thus the two possible interpretations coincide: $(X+\pi U)/\pi U\simeq X/\pi X$. 
   \begin{lemma}\label{Ext=0}
    Let $U$ be a finitely generated $RG$-lattice and suppose that $\Ext_{RG}^1(U,U)=0$. If there is a decomposition of $\overline{U}=X'\oplus Y'$ as $\F G$-module, then there exists a decomposition $U=X\oplus Y$ as $RG$-module with  $\overline{X}=X'$ and $\overline{Y}=Y'$.
\end{lemma}
\begin{proof}
    The exact sequence
    $$\xymatrix{0\ar[r]&U\ar[r]^{\pi}&U\ar[r]&\overline{U}\ar[r]&0},$$
    induces the exact sequence 
    $$\xymatrix{0\ar[r]& \Hom(U,U)\ar[r]^{\pi}&\Hom(U,U)\ar[r]&\Hom(U,\overline{U})\ar[r]&\Ext_{RG}^1(U,U)}=0,$$
    so $\overline{\text{End(U)}}\simeq \text{End}(\overline{U})$. Let $\alpha'$ and $\beta'$ be the projections on $X'$ and $Y'$ respectively with $\alpha'+ \beta'=id_{\overline{U}}$.  Now, $\text{End}(U)$ is a finitely generated algebra over a complete discrete valuation ring, so we can find idempotents $\alpha$ and $\beta$ in $\text{End}(U)$ with $\alpha + \beta=id_U$  satisfying  $\overline{\alpha}=\alpha'$ and $\overline{\beta}=\beta'$ \cite[Theorem 1.9.4]{benson}.  Accordingly, $U$ is a direct sum of
 $RG$-modules $X \oplus Y$ with $\overline{X} = X',\overline{Y} = Y'$. 
\end{proof}

\begin{lemma}(\cite[Corollary 6.8]{MACQUARRIE2020106925}\label{permuExt=0})
    Let $G$ be a finite $p$-group. If $A$ and $B$ are $RG$-permutation modules then $\Ext_{RG}^1(A,B)=0$. 
\end{lemma}

   Let $N$ be a normal subgroup of $G$. Recall that the $RG$-module of  $N$-invariants of $U$ is the largest submodule of $U$ on which $N$ acts trivially, and $U_N$ is the largest $RG$-quotient module of $U$ on which $N$ acts trivially. Denote by $I_N$ the kernel of the augmentation map $RN\longrightarrow R$.  Explicitly, we have:
   $$U^N:=\{u \in U: nu=u \forall n \in N\},$$
   $$U_N:=U/I_NU.$$
    Let us denote by $\widehat{H}\in RH$  the sum 
   of all  elements of a  finite group $H$.   
   
\begin{lemma}\label{free}
    Let $G$ be a cyclic $p$-group and let  $U$ be a finitely generated $\F G$-module such that $U^G=\widehat{G}U$. Then $U$ is $\F G$-free.
\end{lemma}
\begin{proof}
   For a cyclic $p$-group  $G$, there are finitely many isomorphism classes of finitely generated indecomposable modules for the  algebra $\F G$, namely, the ideals  of $\F G$ \cite[Theorem 6.1.2]{peter}. Since the proper ideals of $\F G$ are ideals that lie in the augmentation ideal  of $\F G$, they are annihilated by the multiplication of $\widehat{G}$. Now,  every $\F G$-module has $G$- fixed points,  hence  the only finitely generated $\F G$-module that has $U^G=\widehat{G}U$ must be  $\F G$-free.
 \end{proof}

   From now on, $G$ is a cyclic $p$-group and  $F$ denote  an $RN$-permutation module that does not have trivial direct summands. By $T$ we denote  an $RN$-lattice on which $N$ acts trivially; we will say that $T$ is $RN$-trivial.  We say that $T$ is a maximal trivial summand of $U\res{N}$ --- the restriction of $U$ to $N$ --- if $U\res{N}=S\oplus T$, where $S$ does not have trivial direct summands. Whenever $U\res{N}=S\oplus T$, we will  write $T$ for the module $(T+I_NU)/I_NU$, because, in this case $U_N=T_N\oplus S_N\simeq T\oplus S_N$ as $RN$-modules.

\begin{lemma}\label{G-inva}
     Let $U$ be an $RG$-lattice and $N\leqslant G$ such that $U\res{N}$ is an $RN$-permutation module. If $T$ is a maximal trivial summand of $U\res{N}$, then the image of the natural map $U^N\longrightarrow U_N\longrightarrow\overline{U_N}$ is $\overline{T}$. In particular, the submodule $\overline{T}$ of $\overline{U_N}$ does not depend on the decomposition of $U\res{N}$.
\end{lemma}
\begin{proof}
    Write $U\res{N}=F\oplus T$ where $T$ is maximal trivial. Then $U^N=F^N\oplus T$ and the summands of $F^N$ are of the form $\widehat{N/M}(R[N/M])$ for subgroups $M$ of $N$. But these  summands go to zero in $\overline{U_N}$, hence $\overline{T}$ is the image of $U^N$ in $\overline{U_N}$.
\end{proof}
   Remember that for a cyclic $p$-group $G$ and for an $RG$-module $U$, we have $$H^1(G,U)=\Ker(\widehat{G})/I_GU$$ \cite[Corollary 3.5.2]{benson}. In particular, $U_G$ is a lattice if, and only if, $H^1(G,U)=0$. We will use this observation in what follows.

\begin{lemma}\label{G/M-free}
    Let $U$ be an $RG$-lattice and $N\leqslant G$ such that $U_N$ and $(U^N)_G$ are $RG$-lattices and $U^G=\widehat{G/N}U^N$. Then  $\overline{(U^N/\widehat{N}U)}$ is $\F[G/N]$-free.
\end{lemma}
\begin{proof}
     We will show that there is a natural surjective map $$\xymatrix{U^G\ar[r]&(U^N/\widehat{N}U)^G\ar[r]&\overline{(U^N/\widehat{N}U)}^G}.$$
     Since $U^G=\widehat{G/N}U^N$,  we must have $\overline{(U^N/\widehat{N}U)}^G=\widehat{G/N}\overline{(U^N/\widehat{N}U)}$. But with these conditions,  Lemma \ref{free} says that  $\overline{(U^N/\widehat{N}U)}$ must be $\F[G/N]$-free. To show this, first we should show that $H^1(G,NU)=H^1(G/N,U^N/\widehat{N}U)=0$, because in this case the exact sequences:
     $$\xymatrix{0\ar[r]&\widehat{N}U\ar[r]&U^N\ar[r]&U^N/\widehat{N}U\ar[r]&0}$$ 
     $$\xymatrix{0\ar[r]&U^N/\widehat{N}U\ar[r]^{\pi}&U^N/\widehat{N}U\ar[r]&\overline{U^N/\widehat{N}U}\ar[r]&0}$$
     will provide us a surjective map: $$\xymatrix{U^G\ar@{->>}[r]&(U^N/\widehat{N}U)^G\ar@{->>}[r]&\overline{(U^N/\widehat{N}U)}^G}.$$ To see that $H^1(G,\widehat{N}U)=0$, we observe that  $(U^N)_G$ and $U_N$ are lattices by hypothesis, so $H^1(G/N,U^N)=0$ and $H^1(N,U)=0$. So, from the inflation-restriction  sequence \cite[Theorem 9.84]{Rotman} we get the exact sequence  $$\xymatrix{0\ar[r]&H^1(G/N,U^N)\ar[r]&H^1(G,U)\ar[r]&H^1(N,U)^{G/N}}$$
      that shows that $H^1(G,U)=0$, so $U_G= (U_N)_{G/N}$ is a lattice, therefore $$H^1(G/N,U_N)=0.$$ Since $U_N\simeq \widehat{N}U$  we have $H^1(G/N,\widehat{N}U)=0$. To show that $$H^1(G/N,U^N/\widehat{N}U)=0,$$ 
     denote by $W$ the module $U^N/\widehat{N}U$. Denote by $K$ the kernel of the map defined on $W$  by  multiplication by $\widehat{G/N}$, and let $g$ be a generator of $G$. Then we have $H^1(G/N,W)=K/(g-1)W$. Now, if $u\in U^N$ and we have $\widehat{G/N}u=\widehat{N}u_1$ for some $u_1\in U$,  then the image of $u$ in $W$ lies in $K$ and 
     $$\widehat{N}((g-1)u_1)=0\Longrightarrow (g-1)u_1\in \Kern(\widehat{N})= I_NU=\widehat{G/N}(g-1)U,$$
     because $U_N$ is a lattice.
     It follows that $(g-1)u_1=\widehat{G/N}(g-1)u_2$ with $u_2\in U$, hence $u_1=\widehat{G/N}u_2+ u_3$ with $u_3\in U^G $, since $U^G=\widehat{G/N}U^N$ we get $u_3=\widehat{G/N}u_4$ with  $u_4\in U^N$. From this we have $$\widehat{G/N}u=\widehat{N}u_1=\widehat{N}(\widehat{G/N}u_2+ \widehat{G/N}u_4)=\widehat{G/N}(\widehat{N}u_2+|N|u_4)\Longrightarrow$$ $$\Longrightarrow u=\widehat{N}u_2+|N|u_4+(g-1)u_5,$$ with $u_5\in U^N$ because $(U^N)_G$ being a lattice  means that the kernel of map defined by multiplication by $\widehat{G/N}$ in $U^N$ is $I_GU^N$. So we get $u\in \widehat{N}U+|N|U^N+I_GU^N=\widehat{N}U+I_GU^N$, 
     therefore $K=(\widehat{N}U+I_GU^N)/\widehat{N}U$ and $K/(g-1)W=(\widehat{N}U+I_GU^N)/(\widehat{N}U+I_GU^N)=0$.

\end{proof}
 
\begin{lemma}\label{decomposi}
    Suppose $|G|=p^2$  and let $N$ be the subgroup of $G$ of order $p$. Suppose that $U$ is an $RG$-lattice such that: 
    \begin{enumerate}
        \item 
    $U\downarrow_N$
    is an $RN$-permutation lattice and $U^N=Y\oplus X$ is a decomposition of $U^N$ as an $RG$-module with  $X$ being a maximal trivial summand,
    \item the modules  $U_N$, $U_G$ are $RG$-lattices.
 \end{enumerate}
 Then there exists a decomposition $U\res{N}=F\oplus Z\oplus X$  with $X\oplus Z$ being an $RN$-maximal trivial summand.
 \end{lemma}
\begin{proof}
     Suppose $U^N=Y\oplus X$ where  $X$ is a  $G$-maximal trivial summand.  We claim that $X\cap \widehat{N}U=pX$. Indeed, $$x=\widehat{N}u\in X\cap \widehat{N}U\Longrightarrow  \widehat{N}(g-1)u=0\Longrightarrow (g-1)u\in \ker(\widehat{N}).$$ As $U_N$ is an $RG$-lattice, we have  $$\Ker(\widehat{N})=(g^p-1)U.$$ Note that $\widehat{G/N}(g-1)=g^p-1$, so we get  $$ u=\widehat{G/N}u_1+u_2,$$
     with $u_1\in U,$ and $u_2\in U^G$. Now, $\widehat{N}u_1=y_1+ x_1 $ and  $u_2=y_2+ x_2$ for some $x_i\in X$ and $y_i\in Y$.  It follows that $$x=\widehat{N}u=\widehat{N}(\widehat{G/N}u_1+u_2)=\widehat{G/N}y_1+px_1+py_2+px_2.$$ Since $U^N=Y\oplus X$ we have $x=p(x_1+x_2)\in pX$. Clearly, $pX\leqslant \widehat{N}U$, thus $pX=X\cap \widehat{N}U$. We will now  find the submodule $Z$. First, denote by $\sigma$ the restriction to $X$ of the natural projection  $\gamma:U\longrightarrow T $ relative to the decomposition $U\res{N}=F\oplus T$. We have $\sigma$ injective because if $x\in X\cap F=X\cap F^N$ then $x=\widehat{N}f$ with $f\in F$, and it follows that $x\in X\cap\widehat{N}U=pX$, so $x=px'\Longrightarrow\sigma(x)=p\sigma(x')=0 \Longrightarrow x'\in X\cap F^N\leqslant pX$. Continuing this process, we will get $x\in \bigcap_{i=1}^{\infty}p^iX\subseteq \bigcap_{i=1}^{\infty}\pi^iX=0$. We also have $\sigma(X)\cap \pi T=\pi\sigma(X)$. In fact, if $\pi t\in \sigma(X)$ then there exists $x\in X$ with $x=\widehat{N}f+\pi t $ 
 where  $f\in F$ and $t\in T$, because $U^N=\widehat{N}F\oplus T$. It follows that $$ \pi ^{-1}px=\pi^{-1}p(\widehat{N}f+\pi t)=\widehat{N}\pi^{-1}pf+pt=\widehat{N}(\pi^{-1}pf+ t)\in X\cap \widehat{N}U=pX.$$ From this, it follows that there exists  $x_1=\widehat{N}f_1 + t_1\in X$ with $f_1\in F$ and $t_1\in T$, such that $\pi^{-1}px=px_1$. So we get  $$\pi ^{-1}px=px_1=p(\widehat{N}f_1 + t_1),$$ therefore,
 $$\widehat{N}(\pi^{-1}pf+ t)=\pi^{-1}px=p(\widehat{N}f_1+t_1)\Longrightarrow$$
 $$ \Longrightarrow pt=pt_1 \Longrightarrow t=t_1\Longrightarrow  x_1=\widehat{N}f_1 + t,$$ and $\sigma(x_1)=t\in \sigma(X)$. Thus,  $\pi t \in \pi\sigma(X)$ and we get $\pi\sigma(X)=\sigma(X)\cap \pi T$. Then, since $R$ is a discrete valuation ring, we can write $T=\sigma(X)\oplus Z$ for some $R$-submodule $Z\leqslant T$. Furthermore, since for every $x\in X$ there exists $f\in F$ such that $x=f+\sigma(x)$, we get $\sigma(X)\leqslant F+ X$. From this, we get  $U=F\oplus\sigma(X)\oplus Z=F+X+ Z$. Note that $Z\cap X=0$ because $\sigma(X)\cap Z=0$ and $\sigma$ is a restriction of a map that is the identity on $T$. To finish,  we observe that $F\cap(X\oplus Z)=0$, because if $f\in F$ and $f=x+z$ with $x\in X$ and $z\in Z$, then $0=\gamma(f)=\gamma(x+z)\Longrightarrow \sigma(x)=\gamma(x)=-\gamma(z)=-z$, but $\sigma(X)\cap Z=0$, so $z=x=0$. Therefore, $U\res{N}=F\oplus Z\oplus X$.
\end{proof}
\section{Proof of main theorems}\label{s3}

\begin{proof}[Proof of Theorem \ref{myresult}]
    We argue by induction in the order of $G$ and $\text{rank}(U)$.
    If $|G|=p$ then the theorem is obvious. If $\text{rank}(U)=1$ then a generator $g$ of $G$ acts on $U$ as multiplication by a $p$th root of unity. Since $U\res{N}$ is a permutation module, $N$ acts trivially on $U$. But then $U^N=U$ is an $RG$-permutation module.
     Suppose $|G|>p$ and $\text{rank}(U)>1$. First, we observe that we can assume that $|N|=p$. In fact, if $L$ is the subgroup of $G$ with order $p$  then $U\res{L}$ is an $RL$-permutation module because $U\res{N}$ is a permutation module. We also have that $(U^L)\res{N/L}$ is an $R[N/L]$-permutation module because $U\res{N}$ is an $RN$-permutation module. Now, $(U^L)^{N/L}=U^N$ is an $RG$-permutation module by hypothesis, hence $U^L$ is an $R[G/L]$-permutation module by the induction hypothesis on $|G|$. So, the subgroup $L$ of order $p$ must satisfy the same conditions as $N$, that is, $U\res{L}$ is an $RL$-permutation module and $U^L$ is an $R[G/L]$-permutation module.  We will divide the proof in two cases:
    \begin{itemize}
        \item $|G|=p^2$: As $U^N$ is an $R[G/N]$ permutation module, then  the summands of $U^N$  are either all free, or  some of them are trivial.  First, suppose that $U^N$ is $G/N$-free. Then we have $U^G=\widehat{G/N}U^N$, and by hypothesis $U_N$ and $(U^N)_G$ are lattices, so by Lemma \ref{G/M-free} we get that $\overline{(U^N/\widehat{N}U)}$ is  $\F[G/N]$-free.
      If $U\res{N}$ is $RN$-free then by Theorem \ref{tjpp}, $U$ is an $RG$ permutation module since $U^N$ is an $RG$-permutation module. Now, write $U\downarrow_N=F\oplus T$ where $F$ is $RN$-free and $T\neq 0$ is $N$-trivial. Observe that $\widehat{N}U=\widehat{N}F\oplus pT$, so the image of $\widehat{N}U$ in $\overline{U^N}$ is $\widehat{N}\overline{F}=(\widehat{N}F+\pi U^N)/\pi U^N$. As $U^N=\widehat{N}F\oplus T$, we have $U^N/\widehat{N}U\simeq T/pT$ as $R$-modules. So we have an exact sequence
     $$\xymatrix{0\ar[r]&\overline{\widehat{N}F}\ar[r]&\overline{U^N}\ar[r]&\overline{U^N/\widehat{N}U}\ar[r]&0}.$$ Since $\overline{U^N/\widehat{N}U}$ is $\F[G/N]$-free, the sequence splits, so there exists a submodule $X'$ of  $\overline{U^N}$ such that $\overline{U^N}=\overline{\widehat{N}F}\oplus X'$. Since $U^N$ is an $RG$-permutation module, by Lemmas \ref{Ext=0} and \ref{permuExt=0} there exists a decomposition $U^N=Y\oplus X$ such that $\overline{X}=X'$ and $\overline{Y}=\overline{\widehat{N}F}$. From this, it follows that $U^N=\widehat{N}F\oplus X$ as $RN$-modules, and in particular $X\cap F=0$. So, the image of $X$ in $\overline{U_N}$ is $\overline{T}$ and since $\overline{U_N}=\overline{F_N}\oplus \overline{T}
     $,  we have $U\res{N}=F\oplus X + \pi U+ I_NU$. By Nakayama's Lemma we have $U\res{N}=F\oplus X$. Consider the exact sequence $$\xymatrix{0\ar[r]&X\ar[r]&U\ar[r]&U/X\ar[r]&0}.$$ We have $(U/X)^N=U^N/X=Y$ and $(U/X)\res{N}=F$ because  $U\res{N}=F\oplus X$. As $Y$ is a direct summand of the  $RG$-permutation module $U^N$, then $Y$ is an $RG$-permutation module \cite[Theorem 7.1]{Gregory}. Therefore, by Theorem \ref{tjpp}, we have $U/X$ an $RG$-permutation module. Since $X$ is an $RG$-permutation module, the previous sequence splits over $G$ by Lemma \ref{permuExt=0}, and we have $U\simeq X\oplus U/X$ an $RG$-permutation module.

      Now, let us consider the case where  $U^N=Y\oplus X$ with $Y$ being $R[G/N]$-free and $X$ being $R[G/N]$-trivial. Then by Lemma \ref{decomposi} we can write $U\res{N}=F\oplus Z\oplus X$ with $X\oplus Z$ being an $RN$-maximal trivial summand.  
        Consider the exact sequence
     $$\xymatrix{0\ar[r]&X\ar[r]&U\ar[r]& U/X\ar[r]&0}.$$
     We have $(U/X)\res{N}=F\oplus Z$ and $(U/X)^N=U^N/X=Y$ because $U\res{N}=F\oplus Z\oplus X$. Hence by the  previous argument ($U^N/X=Y$ is $R[G/N]$-free)  $U/X$ is  an $RG$-permutation module. Since $X$ is an $RG$-permutation module, we have $U\simeq X\oplus U/X$ an $RG$-permutation module.

    \item $|G|>p^2$:
     Suppose that  $M$ is a proper subgroup of $G$ containing $N$ properly. Then $U\res{M}$ is an $RM$-permutation lattice by the induction hypothesis on $|G|$ because $U^N\res{M}$ is an $RM$-permutation module and $U\res{N}$ is an $RN$-permutation module. If $U\res{N}$ is $RN$-free, then the result is true by Theorem \ref{tjpp}. So suppose $U\res{N}=F\oplus T$ with $F$ free and $T$ trivial. If $U\res{M}=F_1\oplus T_1$ has a non zero maximal trivial summand $T_1$, then this trivial summand is also an $RM$-maximal  trivial summand of $U^N$. Indeed, we have   $(U^N)\res{M}=F_1^N\oplus T_1$ and $F_1^N$ does not have $RM$-trivial summands -- for the permutation module $R[G/H]$, we have $(R[M/H])^N\simeq R[M/HN]$ for any $H\leqslant M$. Since $U^N$  is an $RG$-permutation module, we can find an $RM$-maximal trivial summand $X$ of $U^N$ which  is $G$-invariant, that is, we can write  $U^N=Y\oplus X$. As $X$ is a maximal trivial summand, by Lemma \ref{G-inva} the image of $Y^M$ in $\overline{(U^N)_M}$ is zero, thus   the images of $X$ and $T_1$ coincide on $\overline{(U^N)_M}$. Therefore, we can write  $U^N=F_1^N\oplus X$ and, consequently, $X\cap F_1=0$. It follows that the images of $X$ and $T_1$ coincide on $\overline{U_M}$, and therefore $U=F_1\oplus X+\pi U+I_MU$. By Nakayama's Lemma, we have $U\res{M}=F_1\oplus X $. Now, consider the exact sequence of $RG$-lattices:
    $$\xymatrix{0\ar[r]&X\ar[r]&U\ar[r]&U/X\ar[r]&0}.$$
     This sequence splits over $M$ because  $U\res{M}=F_1\oplus X $, so  $(U/X)^M=U^M/X=Y^M$. Since $Y$ is an $RG$-permutation module, so is $Y^M$. Then by the induction hypothesis on $\text{rank}(U)$ we have $U/X$ a permutation module. By Lemma \ref{permuExt=0}, the exact sequence  
     $$\xymatrix{0\ar[r]&X\ar[r]&U\ar[r]&U/X\ar[r]&0}$$
     splits over $G$, consequently $U\simeq X\oplus U/X$ is an $RG$-permutation module. Now, if for every proper subgroup $M$ of $G$ which properly contains $N$, $U\res{M}$ does not have an $RM$-maximal trivial summand, then $U^N$ is $G/N$-free because an $RM$-maximal trivial summand of $U\res{M}$ is an $RM$-maximal  trivial summand  of $U^N$ and vice-versa, since $U\res{M}$ is an $RM$ permutation module by the induction hypothesis. So we argue like in the first part of the proof of the case $|G|=p^2$: By the hypothesis and by Lemma \ref{G/M-free}  we get  that $\overline{(U^N/\widehat{N}U)}$  is $\F[G/N]$-free.

      Write $U\downarrow_N=F\oplus T$ where $F$ is $RN$-free and $T$ is $N$-trivial. Observe that $\widehat{N}U=\widehat{N}F\oplus pT$, so the image of $\widehat{N}U$ in $\overline{U^N}$ is $\widehat{N}\overline{F}=(\widehat{N}F+\pi U^N)/\pi U^N$. As $U^N=\widehat{N}F\oplus T$ we have $U^N/\widehat{N}U\simeq T/pT$ as $R$-modules. So, we have an exact sequence
     $$\xymatrix{0\ar[r]&\overline{\widehat{N}F}\ar[r]&\overline{U^N}\ar[r]&\overline{U^N/\widehat{N}U}\ar[r]&0}.$$ Since $\overline{U^N/\widehat{N}U}$ is $\F[G/N]$-free, the sequence splits. Then there exists a submodule $X'$ of  $\overline{U^N}$ such that $\overline{U^N}=\overline{\widehat{N}F}\oplus X'$. Since $U^N$ is an $RG$-permutation module, there exists a decomposition $U^N=Y\oplus X$ such that $\overline{X}=X'$ and $\overline{Y}=\overline{\widehat{N}F}$. From this, it follows that $U^N=\widehat{N}F\oplus X$ as $RN$-modules, in particular $X\cap F=0$. So the image of $X$ in $\overline{U_N}$ is $\overline{T}$, and  since $\overline{U_N}=\overline{F_N}\oplus \overline{T}
     $  we have $U\res{N}=F\oplus X + \pi U+ I_NU$. By Nakayama's Lemma we have $U\res{N}=F\oplus X$. Consider the exact sequence $$\xymatrix{0\ar[r]&X\ar[r]&U\ar[r]&U/X\ar[r]&0}.$$ Then $(U/X)^N=U^N/X=Y$ and $(U/X)\res{N}=F$ because $H^1(N,X)=0$ and $U\res{N}=F\oplus X$. By Theorem \ref{tjpp}, we have $U/X$ an $RG$-permutation module. Now, being  $X$ an $RG$-permutation module, the previous sequence splits over $G$ and we get $U\simeq X\oplus (U/X)$ an $RG$-permutation module.
    \end{itemize} 
\end{proof}
\begin{proof}[Proof of Theorem \ref{corollo}]
    We establish the result by induction on $|G|$. If order $G$ is $p$, then  $U$ is permutation module  because by hypothesis $RC_p$ coflasque modules are $RC_p$-permutation modules. Now, suppose $|G|>p$ and let $N$ be the subgroup of $G$ of order $p$. Then $U\res{N}$ is an $RN$-permutation module since $H^1(N,U)=0$ and by hypothesis $RC_p$-coflasque modules are $RC_p$-permutation modules. Now, let $M$ be a subgroup of $G$  that contains $N$. From the inflation-restriction sequence we get the exact sequence 
    $$\xymatrix{0\ar[r]&H^1(M/N,U^N)\ar[r]&H^1(M,U)\ar[r]&H^1(N,U)^{G/N}}.$$ Since $H^1(M,U)=0$ by hypothesis, it follows  from the exact sequence that $$H^1(M/N,U^N)=0.$$ So  the module  $U^N$ is  $G/N$-coflasque, hence  $U^N$ is an $R[G/N]$-permutation module by  induction hypothesis. Therefore,  the lattice $U$ satisfies the conditions of  Theorem \ref{myresult},  and we conclude that $U$ is an $RG$-permutation module.
    
\end{proof}

\bibliography{ref}

\begin{thebibliography}{10}

\bibitem{benson}
D.~J. Benson.
\newblock {\em Representations and Cohomology. {I}: Basic Representation Theory of Finite Groups and Associative Algebras}.
\newblock Cambridge University Press, 1995.

\bibitem{BOLTJE202070}
R.~Boltje, R.~Kessar, and M.~Linckelmann.
\newblock On {P}icard groups of blocks of finite groups.
\newblock {\em J. Algebra}, 558:70--101, 2020.
\newblock Special Issue in honor of Michel Broué.

\bibitem{butler2}
M.C.R. Butler.
\newblock On the classification of local integral representations of finite abelian $p$-groups.
\newblock In {\em Representations of Algebras}, pages 54--70. Springer, 1974.

\bibitem{Florian}
F.~Eisele.
\newblock On the geometry of lattices and finiteness of {P}icard groups.
\newblock {\em J. reine angew. Math.}, 782:219--233, 2022.

\bibitem{Gregory}
G.~Karpilovsky.
\newblock {\em Induced Modules Over Group Algebras}.
\newblock North-Holland, 1 edition, 1990.

\bibitem{LIVESEY202194}
M.~Livesey.
\newblock On {P}icard groups of blocks with normal defect groups.
\newblock {\em J. Algebra}, 566:94--118, 2021.

\bibitem{MACQUARRIE2020106925}
J.~W. MacQuarrie, P.~Symonds, and P.~A. Zalesskii.
\newblock Infinitely generated pseudocompact modules for finite groups and {W}eiss' theorem.
\newblock {\em Adv. Math.}, 361:106925, 2020.

\bibitem{John}
J.~W. MacQuarrie and P.A. Zalesskii.
\newblock A characterization of permutation modules extending a theorem of {W}eiss.
\newblock {\em Doc. Math}, 25:1159--1169, 2020.

\bibitem{Puig}
L.~Puig.
\newblock {\em On the local structure of Morita and Rickard equivalences between Brauer blocks}.
\newblock Birkhäuser Verlag, 1999.

\bibitem{Rotman}
J.~J. Rotman.
\newblock {\em An Introduction to Homological Algebra}.
\newblock Universitext. Springer, 2 edition, 2009.

\bibitem{TORRECILLAS2013533}
B.~Torrecillas and Th. Weigel.
\newblock Lattices and cohomological mackey functors for finite cyclic $p$-groups.
\newblock {\em Adv. Math.}, 244:533--569, 2013.

\bibitem{peter}
P.~Webb.
\newblock {\em A course in finite group representation theory}.
\newblock Cambridge University Press, 1 edition, 2016.

\bibitem{WeissAnnals}
A.~Weiss.
\newblock Rigidity of {$p$}-adic {$p$}-torsion.
\newblock {\em Ann. of Math. (2)}, 127(2):317--332, 1988.

\bibitem{P}
P.~A. Zalesskii.
\newblock Infinitely generated virtually free pro-$p$ groups and $p$-adic representations.
\newblock {\em J. Topol}, 12(1):79–93, 2019.

\end{thebibliography}

\bibliographystyle{plain} 



\end{document}